\documentclass[A4]{amsart}
\usepackage[square,sort,comma,numbers]{natbib}
\usepackage{amsmath, amssymb, amstext, amsfonts, textcomp, amsxtra, amsbsy, amsgen, amsopn, amscd, mathrsfs, amsthm, latexsym, array}
\usepackage[textwidth=16cm,textheight=22cm,centering]{geometry}

\usepackage{bbm, color}
\usepackage{graphicx}

\usepackage[all]{xy}

\newtheorem{theorem}{Theorem}  [section]
\newtheorem{definition} [theorem] {Definition}
\newtheorem{lemma}[theorem]{Lemma}

\newtheorem{proposition}[theorem]{Proposition}
\newtheorem{remark} [theorem] {Remark}
\newtheorem{conjecture} [theorem] {Conjecture}

\numberwithin{equation}{section} \everymath{\displaystyle}

\newcommand{\Cont}{{\rm C}}

\newcommand{\Tr}{{\rm Tr}}

\newcommand{\Ram}{{\rm Ram}}

\newcommand{\id}{\mathbbm{1}}
\newcommand{\gp}[1]{\mathbf{#1}}
\newcommand{\GL}{{\rm GL}}
\newcommand{\PGL}{{\rm PGL}}
\newcommand{\PD}{{\rm P}\mathcal{D}^{\times}}
\newcommand{\SL}{{\rm SL}}
\newcommand{\PSL}{{\rm PSL}}
\newcommand{\gO}{{\rm O}}
\newcommand{\SO}{{\rm SO}}

\newcommand{\ud}{\mathrm{d}}
\newcommand{\ag}[1]{\mathbb{#1}}

\newcommand{\Z}{\mathbb{Z}}
\newcommand{\B}{\mathcal{B}}
\newcommand{\D}{\mathcal{D}}
\newcommand{\Mat}{{\rm M}}

\newcommand{\Q}{\mathbb{Q}}
\newcommand{\R}{\mathbb{R}}
\newcommand{\C}{\mathbb{C}}
\newcommand{\E}{\mathbf{E}}
\newcommand{\F}{\mathbf{F}}

\newcommand{\A}{\mathbb{A}}

\newcommand{\vO}{\mathcal{O}}
\newcommand{\vo}{\mathfrak{o}}
\newcommand{\VP}{\mathfrak{P}}
\newcommand{\vP}{\mathcal{P}}
\newcommand{\vp}{\mathfrak{p}}

\newcommand{\idlN}{\mathfrak{N}}
\newcommand{\oJ}{\mathfrak{J}}
\newcommand{\oM}{\mathfrak{M}}

\newcommand{\norm}[1][\cdot]{\lvert #1 \rvert}
\newcommand{\extnorm}[1]{\left\lvert #1 \right\rvert}

\newcommand{\fin}{{\rm fin}}

\newcommand{\JL}{\mathbf{JL}}

\newcommand{\Vol}{{\rm Vol}}
\makeatletter
\newcommand{\rmnum}[1]{\romannumeral #1}
\newcommand{\Rmnum}[1]{\expandafter\@slowromancap\romannumeral #1@}
\makeatother

\newcommand{\Op}{{\rm Op}}

\title{Prime Geodesic Theorem for Arithmetic Compact Surfaces: Principal Congruence Case}

\author{Chenhao Tang}
\address{Morningside Center of Mathematics, Academy of Mathematics and Systems Science, Chinese Academy of Sciences; University of the Chinese Academy of Sciences, Beijing 100190, P.R. China}
\email{tangchenhao25@mails.ucas.ac.cn}

\author{Han Wu}
\address{Hangzhou International Innovation Institute, Beihang University, Hangzhou 311115, P.R.China}
\email{forrestwu@buaa.edu.cn}

\author{Jie Yang}
\address{Qiuzhen College, Tsinghua University, 100084, Beijing, P.R. China}
\email{yangjie23@mails.tsinghua.edu.cn}

\author{Wenyan Yang}
\address{School of Mathematical Sciences, University of Scinece and Technology of China, 230026 Hefei, P. R. China}
\email{ywy72@mail.ustc.edu.cn}

\begin{document}

\begin{abstract}
	We generalize Koyama's $7/10$ bound of the error term in the prime geodesic theorems to the principal congruence subgroups for quaternion algebras. Our method avoids the spectral side of the Jacquet--Langlands correspondences, and relates the counting function directly to those for the principal congruence subgroups of Eichler orders of level less than one.
\end{abstract}
	\maketitle
	\tableofcontents

\section{Introduction}

	Let $\Gamma$ be a lattice of $\PSL_2(\R)$. The Selberg zeta function is defined for $\Re(s) > 1$ by
	$$ \mathrm{Z}_{\Gamma}(s) = \prod_{[\gamma_0]} \prod_{k=0}^{\infty} \left( 1 - N([\gamma_0])^{-s-k} \right), $$
where $[\gamma_0]$ runs through the set of primitive hyperbolic conjugacy classes in $\Gamma$, and $N([\gamma_0]) = \alpha^2$ if $\gamma_0$ is conjugate to $\pm \begin{pmatrix} \alpha & \\ & \alpha^{-1} \end{pmatrix}$ for some $\alpha > 1$. It is regarded as an analogue of the Riemann zeta function $\zeta(s)$. The conjugacy classes $[\gamma_0]$ are regarded as an analogue of prime numbers. In particular, their distribution satisfies the same asymptotic as the primes
\begin{equation} \label{eq: PGT}
	\pi_{\Gamma}(x) := \extnorm{ \left\{ [\gamma_0] \ \middle| \ N([\gamma_0]) \leq x \right\} } = \mathrm{Li}(x) + E_{\Gamma}(x).
\end{equation}
	for some error term $E_{\Gamma}(x)$ smaller than $\mathrm{Li}(x)$ as $x \to \infty$. An estimation like \eqref{eq: PGT} is called a prime geodesic theorem for $\Gamma$. It is believed that the true order of the error term is $E_{\Gamma}(x) \ll_{\epsilon} x^{\frac{1}{2}+\epsilon}$. However, the reason is different from the primes. Although $\mathrm{Z}_{\Gamma}(s)$ satisfies the Riemann hypothesis in certain cases, it is a meromorphic function of order $2$, hence has many more zeros than $\zeta(s)$. In fact the Riemann hypothesis for $\mathrm{Z}_{\Gamma}(s)$ implies only $E_{\Gamma}(x) \ll_{\epsilon} x^{\frac{3}{4}+\epsilon}$ (see \cite{He76_Survey} for example). Note that the counting function $\pi_{\Gamma}(x)$, like in the case of prime number theorem, is intimately related to another counting function of an analogue of the von-Mangoldt function
\begin{equation} \label{eq: PGTBis}
	\Psi_{\Gamma}(x) := \sum_{k \in \Z_{\geq 1}, [\gamma_0]: N([\gamma_0])^k \leq x} \frac{\log N([\gamma_0])}{1 - N([\gamma_0])^{-k}} = x + \widetilde{E}_{\Gamma}(x).
\end{equation}
	
	In the special case $\Gamma = \PSL_2(\Z)$, many better bounds for $E_{\Gamma}(x)$ are available, e.g. $E_{\Gamma}(x) \ll_{\epsilon} x^{\frac{35}{48}+\epsilon}$ due to Iwaniec \cite{Iw84}, $E_{\Gamma}(x) \ll_{\epsilon} x^{\frac{7}{10}+\epsilon}$ due to Luo--Rudnick--Sarnak \cite{LRS95}, $E_{\Gamma}(x) \ll_{\epsilon} x^{\frac{71}{102}+\epsilon}$ due to Cai \cite{Cai02}, and $E_{\Gamma}(x) \ll_{\epsilon} x^{\frac{25}{36}+\epsilon}$ due to Soundararajan--Young \cite{SY13}. Note that the Luo--Rudnick--Sarnak bound is valid for any congruence subgroup of $\PSL_2(\Z)$, and that the Soundararajan--Young bound is also valid for principal congruence subgroups of $\PSL_2(\Z)$ by a previous work of the second author \cite{CWZ21}. The generalization of these results to arbitrary co-compact lattices has various difficulties. In the case where $\Gamma = \Gamma_{\D}$ is the unit group of a maximal order of a division quaternion algebra $\D$ defined over $\Q$ and unramified at $\infty$, Koyama \cite{Koy98} extended Luo--Rudnick--Sarnak's bound, and keeps the current record in this case.

	The aim of this paper is to generalize Koyama's result to the principal congruence subgroups. Let $\D$ be a division quaternion algebra defined over $\Q$ and unramified at $\infty$, with reduced norm $\nu_{\D}$. Let $\D^{\times}$ be the group of invertible elements in $\D$, viewed as an algebraic group over $\Q$. Fix a maximal order $\vO_{\D}$ of $\D$, as well as a two-sided ideal $\idlN \subset \vO_{\D}$. Define the principal congruence subgroup of $\PSL_2(\R)$ by
	$$ \Gamma_{\D}^1(\idlN) := \left\{ x \ \middle| \ x \in \vO_{\D}^{\times}, \ \nu_{\D}(x) = 1, \ x - \id_{\D} \in \idlN \right\}, \quad \Gamma_{\D}(\idlN) := \{ \pm \id_{\D} \} \cdot \Gamma_{\D}^1(\idlN). $$
\begin{theorem} \label{thm: Main}
	For any $\epsilon > 0$ we have
	$$ \pi_{\Gamma_{\D}(\idlN)}(x) = \mathrm{Li}(x) + O_{\epsilon}(x^{\frac{7}{10}+\epsilon}), \quad \Psi_{\Gamma_{\D}(\idlN)}(x) = x + O_{\epsilon}(x^{\frac{7}{10}+\epsilon}). $$
\end{theorem}

	Koyama's method exploits the Jacquet--Langlands correspondences, showing that the Laplacian eigenvalues on $\Gamma_{\D} \backslash \ag{H}$ are the same as those of new forms on $\Gamma_0(N) \backslash \ag{H}$ for some square-free integer $N$ determined by $\D$. This is an exploitation of the Jacquet--Langlands correspondences on the spectral sides. Its direct generalization to the (principal) congruence subgroups looks tricky. Our method exploits a basic tool in the trace formula proof of the Jacquet--Langlands correspondences \cite{GJ79}, namely the matching of the orbital integrals. Concretely, we have two ingredients in our method. The first is a relation between the counting function $\Psi_{\Gamma}$ and the adelic stable orbital integrals for a class of lattices $\Gamma$ admitting a large normalizer $\widetilde{\Gamma}$. This is established in \S 3. The second is the matching of the adelic stable orbital integrals for a quaternion division algebra and the $2 \times 2$ matrices. This matching is reduced to the relevant local ones, and is a \emph{variant} of the local matching for $\GL_2$ and its inner form given in \cite{GJ79}. The major difficulty is to find the test functions on the $2 \times 2$ matrices side. Our key observation is the following. Since the image of the local Jacquet--Langlands correspondences consists essentially only of supercuspidal representations, while the supercuspidal representations are constructed from chain orders, good candidates on the $2 \times 2$ matrices side should be given by the principal congruence subgroups associated with these chain orders. We indeed find them in Proposition \ref{prop: ArithMatch}.
	
	More interesting than the error term bound is the equation \eqref{eq: PsiRel}, which expresses the counting function $\Psi_{\Gamma_{\D}(\idlN)}(x)$ in terms of the counting functions $\Psi_{\Gamma}(x)$ for the relevant congruences subgroups $\Gamma$ on the $2 \times 2$ matrices side. Further improvement of the error term bound depends only on those for such $\Gamma < \PSL_2(\Z)$. In the case of Soundararajan--Young bound, such an improvement has non-trivial difficulties. We refer the reader to the computation for the Hecke congruence subgroups in \cite{CWZ21}. 
	
	A generalization of our method to co-compact quaternion lattices of $\PSL_2(\C)$ should be easy. It is also interesting to extend our result to non-principal congruence subgroups, which would require a new way of adelization without the crucial assumption \eqref{eq: LocGlobLatGps} below. We formalize such possibility as a conjecture:
\begin{conjecture}
	For any congruence subgroup $\Gamma$ of $\PSL_2(\R)$ defined by $\D$, there exist finitely many congruence subgroups $\Gamma_i < \PSL_2(\Z)$ and $a_i \in \Q^{\times}$ so that
	$$ \Psi_{\Gamma}(x) = \sum_i a_i \cdot \Psi_{\Gamma_i}(x), \quad \sum_i a_i = 1. $$
\end{conjecture}

\section*{Acknowledgement}

	The authors thank the referees for suggestions improving on the readability of the paper. They would like to thank the organizers of the 2025 summer school ``Algebra and Number Theory'' held at the Chinese Academy of Sciences, during which event the paper is written. Wu would like to thank Jiandi Zou for inspiring discussion related to the chain orders. Part of the work is done during Wu's visit at the IASM of Zhejiang University. Wu would also like to thank Prof. Binyong Sun and IASM for hospitality and the support from the New Cornerstone Science Foundation.

\section{A Matching of non-Archimedean Orbital Integrals}

	\subsection{Matching of Orbital Integrals}
	
	Let $\F$ be a $p$-adic field with normalized additive valuation $v$, valuation ring $\vo$, valuation ideal $\vp$ and a chosen uniformizer $\varpi$. Let $q := \norm[\vo/\vp]$. Up to isomorphism, there are two quaternion algebras over $\F$: $\Mat_2(\F)$ and the unique division quaternion algebra $\D(\F)$. Their reduced traces and reduced norms are denoted by $\Tr$, resp. $\Tr_{\D}$ and $\det$, resp. $\nu_{\D}$. Fix a uniformizer $\varpi_{\D}$ of $\D(\F)$ with $\nu_{\D}(\varpi_{\D}) = \varpi$. Their groups of invertible elements are $\Mat_2(\F)^{\times} = \GL_2(\F)$ and $\D(\F)^{\times} = \D^{\times}(\F)$, respectively. We are particularly interested in their quotients by the center, $\PGL_2(\F)$ and $\PD(\F) := \D^{\times}(\F)/\F^{\times}$, respectively.
	
	Let $\mathcal{A}_{\D}(\F)$ and $\mathcal{A}^{\lozenge}(\F)$ be the set of irreducible unitary representations of $\PD(\F)$ and square-integrable representations of $\PGL_2(\F)$, respectively. The following is (a special case of) the (local non-archimedean) \emph{Jacquet--Langlands correspondences} for $\PGL_2$.
	
\begin{theorem} \label{thm: JLCh}
	There is a unique bijective map
	$$ \JL: \mathcal{A}_{\D}(\F) \to \mathcal{A}^{\lozenge}(\F), \quad \pi \mapsto \JL(\pi) $$
	characterized by the following condition:
\begin{itemize}
	\item[(Ch)] Let $\E/\F$ be a quadratic field extension. Let $\theta_*$ be the \emph{character function} of the representation $*$. Then for any $x \in \E^{\times}$ and any $\F$-algebra embeddings $\iota: \E \to \Mat_2(\F)$ and $\iota': \E \to \D(\F)$ we have
	$$ \theta_{\JL(\pi)}(\iota(x)) = - \theta_{\pi}(\iota'(x)), \quad \forall x \in \E^{\times}. $$
\end{itemize}
\end{theorem}
\begin{proof}
	This is \cite[Theorem (8.1)]{GJ79}, up to notational difference.
\end{proof}

\noindent In order to capture the above correspondences in the trace formulae, matching orbital integrals are introduced in \cite[\S 8.B]{GJ79}. We reformulate a \emph{variant} as follows.

\begin{definition} \label{def: OrbIntM}
	Let $f \in \Cont_c^{\infty}(\GL_2(\F))$ and $\varphi \in \Cont_c^{\infty}(\D^{\times}(\F))$. 
%	Define $\widetilde{f} \in \Cont_c^{\infty}(\PGL_2(\F))$ and $\widetilde{\varphi} \in \Cont_c^{\infty}(\PD(\F))$ by the formulae
%	$$ \widetilde{f}(g) := \int_{\F^{\times}} f(zg) \ud^{\times}z, \quad \widetilde{\varphi}(g) := \int_{\F^{\times}} \varphi(zg) \ud^{\times}z. $$

\noindent (1) For any quadratic extension of algebras $\E/\F$ with an embedding $\iota: \E \to \Mat_2(\F)$, resp. $\iota': \E \to \D(\F)$, we introduce the orbital integral for any \emph{regular element} $x$, i.e., elements $x \in \E^{\times}-\F$, by
	$$ \vO(f; \E,x) := \int_{\iota(\E^{\times}) \backslash \GL_2(\F)} f(g^{-1} \iota(x) g) \ud g \quad \text{resp.} \quad \vO(\varphi; \E,x) := \int_{\iota'(\E^{\times}) \backslash \D^{\times}(\F)} \varphi(g^{-1} \iota'(x) g) \ud g, $$
%	$$ \vO(f; \E,x) := \int_{\iota(\E^{\times}) \backslash \GL_2(\F)} f(g^{-1} \iota(x) g) \ud g, \quad \text{resp.} \quad \vO(\varphi; \E,x) := \int_{\iota'(\E^{\times}) \backslash \PD(\F)} \varphi(g^{-1} \iota'(x) g) \ud g. $$
	where we transport a common Haar measure of $\E^{\times}$ to both $\iota(\E^{\times})$ and $\iota'(\E^{\times})$.
	
\noindent (2) The functions $f$ and $\varphi$ are said to be \emph{matching}, denoted by $f \leftrightarrow \varphi$, if:
\begin{itemize}
	\item[(\rmnum{1})] For $\E \simeq \F \oplus \F$ and regular $x$ we have $\vO(f; \E,x) = 0$;
	\item[(\rmnum{2})] For $\E$ field and regular $x$ we have $\vO(f; \E,x) = \vO(\varphi; \E,x)$.
\end{itemize}
\end{definition}

\noindent Note that for matching functions $f$ and $\varphi$ one has (see \cite[(8.5), (8.6) \& (8.9)]{GJ79})
\begin{align} 
	\Tr \left( \JL(\pi)(f) \right) &= - \Tr \left( \pi(\varphi) \right), \ \forall \pi \in \mathcal{A}_{\D}(\F); \\
	\int_{\GL_2(\F)} f(g) \chi(\det g) \ud g &= \int_{\D^{\times}(\F)} \varphi(g) \chi(\nu_{\D}(g)) \ud g, \ \forall \chi \in \widehat{\F^{\times}}; \\
	\Tr(\pi(f)) &= 0
\end{align}
	for irreducible admissible $\pi$ of $\PGL_2(\F)$ that is neither square-integrable nor finite dimensional. 
\begin{remark}
	If $f$ and $\varphi$ have a common central character, then for a given $\varphi$ the existence of a/many matching $f$ is ensured by \cite[Lemma (8.10)]{GJ79}. For our version without central character, such existence is not known. The existence for our version obviously implies the one with central character.
\end{remark}

	\subsection{An Arithmetic Matching}
	
	Chain orders (see \cite[\S 12]{BuH06}) play an important role in the classification of the supercuspidal representations of $\GL_2(\F)$. They also turn out to be important for our purpose.
	
	Precisely, in $\Mat_2(\F)$ we have two chain $\vo$-orders $\oM := \Mat_2(\vo)$ and $\oJ := \begin{pmatrix} \vo & \vo \\ \vp & \vo \end{pmatrix}$. Let $\VP_{\oM} := \varpi \Mat_2(\vo)$, resp. $\VP_{\oJ} = \begin{pmatrix} \vp & \vo \\ \vp & \vp \end{pmatrix}$ be the Jacobson radical of $\oM$, resp. $\oJ$. Let $\gp{K}_{\oM}$, resp. $\gp{K}_{\oJ}$ be the normalizer subgroup of $\oM$, resp. $\oJ$ in $\GL_2(\F)$. Concretely we have 
	$$ \VP_{\oM} = \varpi \oM, \quad \VP_{\oJ} = \Pi \oJ; \qquad \gp{K}_{\oM} = \varpi^{\Z} \oM^{\times}, \quad \gp{K}_{\oJ} = \Pi^{\Z} \oJ^{\times}; \qquad \Pi := \begin{pmatrix} & 1 \\ \varpi & \end{pmatrix}. $$
	Note that up to conjugation $\gp{K}_{\oM}$ and $\gp{K}_{\oJ}$ are all the maximal compact subgroups of $\PGL_2(\F)$. Note also that, since $\Pi^2 = \varpi \id_2$, $\vo^{\times} \id_2 < \oJ^{\times}$ and $\Pi$ normalizes $\oJ$ hence also $\oJ^{\times}$ we have
	$$ \gp{K}_{\oJ} = \Pi^{2\Z} \mathfrak{J}^{\times} \bigsqcup \Pi^{2\Z+1} \mathfrak{J}^{\times} = \varpi^{\Z} \mathfrak{J}^{\times} \bigsqcup \varpi^{\Z} \mathfrak{J}^{\times} \Pi. $$
	
	 In $\D(\F)$ we have the unique maximal $\vo$-order $\vO_{\D} := \left\{ x \in \D(\F) \ \middle| \ \nu_{\D}(x) \in \vo \right\}$. Then $\F^{\times}\vO_{\D}^{\times}$ is the unique maximal compact subgroup of $\PD(\F)$. The Jacobson radical of $\vO_{\D}$ is $\VP_{\vO_{\D}} = \vP_{\D} := \varpi_{\D} \vO_{\D}$. The normalizer subgroup of $\vO_{\D}$ is $\D^{\times}(\F)$.
	 
\begin{definition} \label{def: PrinpChF}
	(1) For $n \in \Z_{\geq 0}$ write $U_{\vo}^n := (1+\vp^n) \cap \vo^{\times}$.

\noindent (2) Let $\vO \in \{ \oM, \oJ, \vO_{\D} \}$ with Jacobson radical $\VP$. The \emph{principal congruence subgroups} for $\vO^{\times}$ are indexed by $n \in \Z_{\geq 0}$ as $U_{\vO}^n := (1 + \VP^n) \cap \vO^{\times}$. We also write $U_{\D}^n = U_{\vO_{\D}}^n$ for simplicity.
	
\noindent (3) We introduce the following normalized characteristic functions
	$$ f_n := \tfrac{1}{[\vo^{\times}:\det U_{\oM}^n]} \cdot \tfrac{1}{\Vol(U_{\oM}^n)} \id_{U_{\oM}^n}, \quad g_n := \tfrac{1}{[\vo^{\times}:\det U_{\oJ}^n]} \cdot \tfrac{1}{\Vol(U_{\oJ}^n)} \id_{U_{\oJ}^n}, \quad \varphi_n := \tfrac{1}{[\vo^{\times}:\nu_{\D} \left( U_{\D}^n \right)]} \cdot \tfrac{1}{\Vol(U_{\D}^n)} \id_{U_{\D}^n}. $$
\end{definition}

\begin{proposition} \label{prop: ArithMatch}
	We have the matching of test functions in the sense of Definition \ref{def: OrbIntM}
	$$ \begin{cases} 
	\widetilde{f}_0 := 2f_0 - g_0 \leftrightarrow \varphi_0 & \\ 
	\widetilde{f}_{2n} := \tfrac{2q}{q-1} f_{n} - \tfrac{q+1}{q-1} g_{2n} \leftrightarrow \varphi_{2n} & \forall \ n \in \Z_{\geq 1} \\
	\widetilde{f}_{2n-1} := - \tfrac{2}{q-1} f_{n} + \tfrac{q+1}{q-1} g_{2n-1} \leftrightarrow \varphi_{2n-1} & \forall \ n \in \Z_{\geq 1}
	\end{cases}. $$
\end{proposition}
\begin{remark}
	It is easy to compute
	$$ \det \left( U_{\oM}^n \right) = U_{\vo}^n, \quad \det \left( U_{\oJ}^n \right) = U_{\vo}^{\lceil \frac{n}{2} \rceil}, \quad \nu_{\D} \left( U_{\D}^n \right) = U_{\vo}^{\lceil \frac{n}{2} \rceil}. $$
	So the factors containing their index in $\vo^{\times}$ are equal on both sides of the matching functions. Hence we can ignore them in the proof of Proposition \ref{prop: ArithMatch}. See Remark \ref{rmk: B1NotB} for more explication.
\end{remark}

	We are going to prove Proposition \ref{prop: ArithMatch} by explicit computation of the relevant orbital integrals. We start with split $\E$.

\begin{lemma} \label{lem: OrbIntSp}
	Consider $\E = \F \oplus \F$ and regular elements $x = (a,b) \in \F^{\times} \oplus \F^{\times}$ with $a \neq b$. Choose $\iota: \E \hookrightarrow \Mat_2(\F)$ to be the diagonal embedding.
	
\noindent (1) We have the decompositions
	$$ \GL_2(\F) = \bigsqcup_{r=0}^{\infty} \iota(\E^{\times}) \begin{pmatrix} 1 & \varpi^{-r} \\ & 1 \end{pmatrix} \gp{K}_{\oM}, \quad \GL_2(\F) = \bigsqcup_{r=0}^{\infty} \iota(\E^{\times}) \begin{pmatrix} 1 & \varpi^{-r} \\ & 1 \end{pmatrix} \gp{K}_{\oJ}. $$
	Moreover, for $\vO \in \{ \oM, \oJ \}$ we have the volume computation for $r \geq 1$
	$$ \frac{\Vol \left( \iota(\E^{\times}) \backslash \iota(\E^{\times}) \begin{pmatrix} 1 & \varpi^{-r} \\ & 1 \end{pmatrix} \gp{K}_{\vO} \right)}{\Vol(\vO^{\times})} = \frac{q^{r-1}(q-1)}{\Vol(\vo^{\times} \times \vo^{\times})} \cdot 
	\begin{cases}
		1 & \text{if } \vO = \oM \\
		2 & \text{if } \vO = \oJ
	\end{cases}. $$

\noindent (2) We have the formulae
	$$ \vO(f_n; \E, (a,b)) = \frac{1}{\Vol(\vo^{\times} \times \vo^{\times})} \cdot \frac{\id_{U_{\vo}^n}(a)\id_{U_{\vo}^n}(b)}{\norm[a-b]_{\F}} \cdot
	\begin{cases}
		1 & \text{if } n = 0 \\
		q^{3n-3}(q-1)^2(q+1) & \text{if } n \geq 1
	\end{cases}.
	$$
	
\noindent (3) We have the formulae
	$$ \vO(g_n; \E, (a,b)) = \frac{2}{\Vol(\vo^{\times} \times \vo^{\times})} \cdot \frac{\id_{U_{\vo}^{\lceil \frac{n}{2} \rceil}}(a) \id_{U_{\vo}^{\lceil \frac{n}{2} \rceil}}(b)}{\norm[a-b]_{\F}} \cdot 
	\begin{cases}
		1 & \text{if } n = 0 \\
		q^{n+\lfloor \frac{n}{2} \rfloor - 2}(q-1)^2 & \text{if } n \geq 1
	\end{cases}.
	$$
\end{lemma}
\begin{proof}
	(1) The first equality follows readily from the usual Iwasawa decomposition. For the second one, we first show that the right hand side is a disjoint union. In fact, for $r_1,r_2 \in \Z_{\geq 0}$ the condition
	$$ \begin{pmatrix} a & a \varpi^{-r_2} - b \varpi^{-r_1} \\ & b \end{pmatrix} = \begin{pmatrix} 1 & - \varpi^{-r_1} \\ & 1 \end{pmatrix} \begin{pmatrix} a & \\ & b \end{pmatrix} \begin{pmatrix} 1 & \varpi^{-r_2} \\ & 1 \end{pmatrix} \in \gp{K}_{\oJ} = \F^{\times} \oJ^{\times} \sqcup \F^{\times} \oJ^{\times} \Pi $$
	is satisfied only if the left hand side lies in $\F^{\times} \oJ^{\times}$, since $\gp{B}(\F) \cap \oJ^{\times} \Pi = \emptyset$, where $\gp{B}(\F)$ is the subgroup of upper triangular matrices of $\GL_2(\F)$. Therefore for some $z \in \F^{\times}$
	$$ \begin{pmatrix} za & za \varpi^{-r_2} - zb \varpi^{-r_1} \\ & zb \end{pmatrix} \in \oJ^{\times} \ \Rightarrow \ za,zb \in \vo^{\times}, za \varpi^{-r_2} - zb \varpi^{-r_1} \in \vo. $$
	The last condition implies $r_1=r_2$ since otherwise $v(za \varpi^{-r_2} - zb \varpi^{-r_1}) = \min \left( v(za\varpi^{-r_2}), v(zb\varpi^{-r_1}) \right) = \min(-r_2,-r_1) < 0$. We then note the coset decomposition
	$$ \oM^{\times} = \bigsqcup_{u \in \vo/\vp} w \begin{pmatrix} 1 & \\ u & 1 \end{pmatrix} \oJ^{\times} \bigsqcup \oJ^{\times}, \quad w := \begin{pmatrix} & 1 \\ 1 & \end{pmatrix}. $$
	We get the second equality from the first one since
\begin{multline*} 
	\iota(\E^{\times}) \begin{pmatrix} 1 & \varpi^{-r} \\ & 1 \end{pmatrix} w \begin{pmatrix} 1 & \\ u & 1 \end{pmatrix} \oJ^{\times} = \iota(\E^{\times}) \begin{pmatrix} 1 & \varpi^{-r-1} \\ & 1 \end{pmatrix} \Pi \begin{pmatrix} 1 & \\ u & 1 \end{pmatrix} \oJ^{\times} \\
	= \iota(\E^{\times}) \begin{pmatrix} 1 & \varpi^{-r-1}+\varpi^{-1}u \\ & 1 \end{pmatrix} \Pi \oJ^{\times} \subset \iota(\E^{\times}) \begin{pmatrix} 1 & \varpi^{-k} \\ & 1 \end{pmatrix} \gp{K}_{\oJ}
\end{multline*}
	for some $k \in \Z_{\geq 0}$. The ``moreover'' part follows readily from
	$$ \iota(\E^{\times}) \cap \begin{pmatrix} 1 & \varpi^{-r} \\ & 1 \end{pmatrix} \gp{K}_{\vO} \begin{pmatrix} 1 & -\varpi^{-r} \\ & 1 \end{pmatrix} = \F^{\times} \left\{ \begin{pmatrix} a & \\ & b \end{pmatrix} \ \middle| \ a,b \in \vo^{\times}, a-b \in 1+\vp^r \right\}, $$
	which has index $q^{r-1}(q-1)$ in $\iota(\E^{\times}) \cap \gp{K}_{\vO}$, $\iota(\E^{\times}) \cap \vO^{\times} = \iota(\vo^{\times} \times \vo^{\times})$ and
	$$ \iota(\E^{\times}) \backslash \iota(\E^{\times}) \gp{K}_{\oM} = \iota(\E^{\times}) \backslash \iota(\E^{\times}) \oM^{\times} \simeq \iota(\vo^{\times} \times \vo^{\times}) \backslash \oM^{\times}, $$
	$$ \iota(\E^{\times}) \backslash \iota(\E^{\times}) \gp{K}_{\oJ} = \left( \iota(\E^{\times}) \backslash \iota(\E^{\times}) \oJ^{\times} \right) \bigsqcup \left( \iota(\E^{\times}) \backslash \iota(\E^{\times}) \oJ^{\times} \Pi \right) \simeq \left( \iota(\vo^{\times} \times \vo^{\times}) \backslash \oM^{\times} \right) \times \Z/2\Z. $$
	
\noindent (2) Since $U_{\oM}^n \triangleleft \gp{K}_{\oM}$, the function $f_n$ is invariant by conjugation by $\gp{K}_{\oM}$. We get by (1)
\begin{multline*}
	\vO(f_n; \E, (a,1)) \\
	= \frac{\Vol \left( \iota(\E^{\times}) \backslash \iota(\E^{\times}) \gp{K}_{\oM} \right)}{\Vol(U_{\oM}^n)} \cdot \left\{ \id_{U_{\oM}^n} \left( \begin{pmatrix} a & \\ & 1 \end{pmatrix} \right) + \sum_{r=1}^{\infty} q^{r-1}(q-1) \cdot \id_{U_{\oM}^n} \left( \begin{pmatrix} a & \varpi^{-r}(a-1) \\ & 1 \end{pmatrix} \right) \right\} \\
	= \frac{\Vol \left( \iota(\E^{\times}) \backslash \iota(\E^{\times}) \gp{K}_{\oM} \right)}{\Vol(U_{\oM}^n)} \cdot \left\{ \id_{U_{\vo}^n} \left( a \right) + \sum_{r=1}^{\infty} q^{r-1}(q-1) \cdot \id_{U_{\vo}^{n+r}} \left( a \right) \right\} \\
	= \frac{\Vol \left( \iota(\E^{\times}) \backslash \iota(\E^{\times}) \gp{K}_{\oM} \right)}{\Vol(U_{\oM}^n)} \cdot \frac{\id_{U_{\vo}^n}(a)}{q^n \cdot \norm[a-1]_{\F}}.
\end{multline*}
	The stated formula follows readily from the indices computation
\begin{equation} \label{eq: MUnitsInd}
	[\oM^{\times} : U_{\oM}^1] = \extnorm{\GL_2(\mathbb{F}_q)} = (q^2-q)(q^2-1), \quad [U_{\oM}^1 : U_{\oM}^n] = q^{4(n-1)}. 
\end{equation}
	
\noindent (3) Since $U_{\oJ}^n \triangleleft \gp{K}_{\oJ}$, the function $g_n$ is invariant by $\gp{K}_{\oJ}$ by conjugation. We get by (1)
\begin{multline*}
	\vO(g_n; \E, (a,b)) \\
	= \frac{\Vol \left( \iota(\E^{\times}) \backslash \iota(\E^{\times}) \gp{K}_{\oJ} \right)}{\Vol(U_{\oJ}^n)} \cdot \left\{ \id_{U_{\oJ}^n} \left( \begin{pmatrix} a & \\ & b \end{pmatrix} \right) + \sum_{r=1}^{\infty} q^{r-1}(q-1) \cdot \id_{U_{\oJ}^n} \left( \begin{pmatrix} a & \varpi^{-r}(a-b) \\ & b \end{pmatrix} \right) \right\} \\
	= \frac{\Vol \left( \iota(\E^{\times}) \backslash \iota(\E^{\times}) \gp{K}_{\oJ} \right)}{\Vol(U_{\oJ}^n)} \cdot 
	\begin{cases}
		\id_{U_{\vo}^k} \left( a \right) \id_{U_{\vo}^k} \left( b \right) \left\{ 1 + \sum_{r=1}^{\infty} q^{r-1}(q-1) \cdot \id_{\vp^{k+r}} \left( a-b \right) \right\} & \text{if } n = 2k \\
		\id_{U_{\vo}^{k+1}} \left( a \right) \id_{U_{\vo}^{k+1}} \left( b \right) \left\{ 1 + \sum_{r=1}^{\infty} q^{r-1}(q-1) \cdot \id_{\vp^{k+r}} \left( a-b \right) \right\} & \text{if } n = 2k+1
	\end{cases} \\
	= \frac{\Vol \left( \iota(\E^{\times}) \backslash \iota(\E^{\times}) \gp{K}_{\oJ} \right)}{\Vol(U_{\oJ}^n)} \cdot \frac{\id_{U_{\vo}^{\lceil \frac{n}{2} \rceil}}(a) \id_{U_{\vo}^{\lceil \frac{n}{2} \rceil}}(b)}{q^{\lfloor \frac{n}{2} \rfloor} \cdot \norm[a-b]_{\F}}.
\end{multline*}
	The stated formula follows readily from the indices computation
\begin{equation} \label{eq: JUnitsInd}
	[\oJ^{\times} : U_{\oJ}^1] = \extnorm{\mathbb{F}_q^{\times}}^2 = (q-1)^2, \quad [U_{\oJ}^1 : U_{\oJ}^n] = q^{2(n-1)}. 
\end{equation}
\end{proof}

	We turn to non-split $\E$. Let $\vO_{\E}$, resp. $\vP_{\E}$ be the valuation ring, resp. valuation ideal, of $\E$. Write $U_{\E}^n := (1+\vP_{\E}^n) \cap \vO_{\E}^{\times}$ for $n \in \Z_{\geq 0}$. Let $e=e(\E/\F)$ be the ramification index of $\E/\F$. We start with $\varphi_n$.
	
\begin{lemma} \label{lem: OrbIntNSD}
	Consider non-split $\E/\F$. We have the following formulae for $n \in \Z_{\geq 0}$
	$$ \vO(\varphi_n; \E, x) = \frac{\Vol \left( \iota'(\E^{\times}) \backslash \D^{\times}(\F) \right)}{\Vol(U_{\D}^n)} \cdot \id_{U_{\E}^{\lceil \frac{en}{2} \rceil}}(x) = \frac{2}{e \Vol(\vO_{\E}^{\times})} \cdot \id_{U_{\E}^{\lceil \frac{en}{2} \rceil}}(x) \cdot 
	\begin{cases}
		1 & \text{if } n = 0 \\
		q^{2n}(1-q^{-2}) & \text{if } n \geq 1
	\end{cases}. $$
\end{lemma}
\begin{proof}
	It suffices to note that $U_{\D}^n \triangleleft \D^{\times}(\F)$ and $\vP_{\D}^n \cap \iota'(\E) = \iota' \left( \vP_{\E}^{\lceil \frac{en}{2} \rceil} \right)$. Details are left to the reader.
\end{proof}

\noindent The case for $f_n$ and $g_n$ requires more arithmetic tools of orders. Recall all quadratic $\vo$-sub-orders of $\vO_{\E}$ are parametrized by $r \in \Z_{\geq 0}$ with $L_r := \vo + \varpi^r \vO_{\E}$ (see \cite[\S 4.1]{CWZ21} for example).

\begin{definition} \label{def: OpEmb}
	Let $B \subset \vO_{\E}$ be an $\vo$-sub-order and $\vO \in \{ \oM, \oJ \}$. An $\F$-embedding $\iota: \E \to \Mat_2(\F)$ is called \emph{optimal} with respect to $B$ and $\vO$ if $\iota(\E) \cap \vO = \iota(B)$. The set of such embeddings is denoted by $\Op(B, \vO)$. It is acted by $\gp{K}_{\vO}$ from right by $\iota^{g}(x) := g^{-1} \iota(x) g$ for any $\iota \in \Op(B, \vO)$ and $g \in \gp{K}_{\vO}$.
\end{definition}

\begin{theorem} \label{thm: OpEmbNr}
	The number of orbits of $\gp{K}_{\vO}$ on $\Op(L_r, \vO)$ is $1$ unless $e=1, r=0$ and $\vO = \oJ$, in which case $\Op(\vO_{\E}, \oJ) = \emptyset$.
\end{theorem}
\begin{proof}
	This is (part of) \cite[Theorem \Rmnum{2}.3.2]{Vi80}.
\end{proof}

\begin{lemma} \label{lem: OpEmbFil}
	(1) There exists $\theta_0 \in \vO_{\E}$ such that $L_r = \vo[\varpi^r \theta_0]$ as $\vo$-algebra.

\noindent (2) Let $\iota \in \Op(L_r, \vO)$ for $\vO \in \{ \oM, \oJ \}$ with Jacobson radical $\VP$. Then we have for $n \in \Z_{\geq 0}$
	$$ \VP^n \cap \iota(\E) = 
	\begin{cases}
		\iota(\varpi^n L_r) & \text{if } \vO = \oM \\
		\iota(\varpi^k L_r) & \text{if } \vO = \oJ \text{ and } n=2k \\
		\iota(\varpi^k L_{r-1}) & \text{if } \vO = \oJ \text{ and } n=2k-1, r \geq 1 \\
		\iota(\vP_{\E}^n) & \text{if } \vO = \oJ \text{ and } n=2k-1, r = 0
	\end{cases}. $$
\end{lemma}
\begin{proof}
	(1) In fact any $\theta_0 \in \vO_{\E}$ satisfying $\vO_{\E} = \vo \oplus \vo \theta_0$ is a good choice.
	
\noindent (2) The formula for $\vO = \oM$ is easy, since $\VP = \varpi \oM$. For $\vO = \oJ$, we first note that the case $n=2k$ is easy, since $\VP^{2k} = \varpi^k \oJ$. We also note that the case $n=2k-1$ follows from the case $n=1$, since $\VP^{2k-1} = \varpi^{k-1} \VP$. We only need to prove the case $n=1$. For simplicity of notation we regard $\E$ as a subset of $\Mat_2(\F)$ via $\iota$. Therefore we have $\oJ \cap \E = L_r$ by assumption. Suppose $\theta_0^2 + a \theta_0 + b = 0$ for some $a,b \in \vo$. Let $f_r(X) = X^2 + a \varpi^r X + b \varpi^{2r}$, which is the minimal polynomial of $\varpi^r \theta_0$. If $r \geq 1$, then
	$$ L_r/\varpi L_r \simeq \vo[X]/(\vp, f_r(X)) \simeq \mathbb{F}_q[X]/(X^2) $$
is a non-reduced ring, i.e., a ring containing non-zero nilpotent elements. Note that there is a unique non-zero ideal in $\mathbb{F}_q[X]/(X^2)$. Hence there is a unique $L_r$-ideal lying between $L_r$ and $\varpi L_r$, which is $\varpi L_{r-1}$. We have $L_r/\varpi L_{r-1} \simeq \mathbb{F}_q$. Now that $\oJ/\VP \simeq \mathbb{F}_q \oplus \mathbb{F}_q$ is reduced, we have $\VP \cap \E \neq \varpi L_r$. Otherwise we would obtain the following embedding of rings, which is absurd:
	$$ \begin{matrix}
		\mathbb{F}_q[X]/(X^2) & \hookrightarrow & \mathbb{F}_q \oplus \mathbb{F}_q \\
		\parallel & & \parallel \\
		L_r/\varpi L_r & \hookrightarrow & \oJ/\VP
	\end{matrix}. $$ 
	We also have $\VP \cap \E \neq L_r$ since $1 \notin \VP$. We conclude $\VP \cap \E = \varpi L_{r-1}$ in this case. Finally for $r=0$, the extension $\E/\F$ must be ramified by Theorem \ref{thm: OpEmbNr}. There is a unique $\vO_{\E}$-ideal lying between $\vO_{\E}$ and $\varpi \vO_{\E} = \vP_{\E}^2$, which is $\vP_{\E}$. Note that $\vO_{\E}/\vP_{\E}^2$ is non-reduced, too. Repeating the previous argument we conclude $\VP \cap \E = \vP_{\E}$.
\end{proof}

\begin{lemma} \label{lem: IwaDVar}
	For $\vO \in \{ \oM, \oJ \}$, a quadratic field $\E/\F$ of ramification index $e$, and any $\F$-embedding $\iota: \E \to \Mat_2(\F)$, there exist $a_r \in \GL_2(\F)$ such that $\iota^{a_r} \in \Op(L_r, \vO)$ and
	$$ \GL_2(\F) = \sideset{}{_{r \geq 0}} \bigsqcup \iota(\E^{\times}) a_r \gp{K}_{\oM}, \quad \GL_2(\F) = \sideset{}{_{r \geq 2-e}} \bigsqcup \iota(\E^{\times}) a_r \gp{K}_{\oJ}. $$
	Moreover, we have the volume computation
	$$ \frac{\Vol \left( \iota(\E^{\times}) \backslash \iota(\E^{\times}) a_r \gp{K}_{\oM} \right)}{\Vol(\oM^{\times})} = \frac{1}{\Vol(L_r^{\times})} = \frac{1}{\Vol(\vO_{\E}^{\times})} \cdot \begin{cases}
		q^r \frac{1-q^{-2}}{1-q^{-e}} & \text{if } r \geq 1 \\
		1 & \text{if } r = 0
	\end{cases}; $$
	$$ \frac{\Vol \left( \iota(\E^{\times}) \backslash \iota(\E^{\times}) a_r \gp{K}_{\oJ} \right)}{\Vol(\oJ^{\times})} = \frac{1}{\Vol(L_r^{\times})} \cdot
	\begin{cases}
		2 \\
		1
	\end{cases} = \frac{1}{\Vol(\vO_{\E}^{\times})} \cdot
	\begin{cases}
		2q^r \frac{1-q^{-2}}{1-q^{-e}} & \text{if } r \geq 1 \\
		1 & \text{if } r = 0
	\end{cases}. $$
\end{lemma}
\begin{proof}
	By Skolem--Noether's theorem any two $\F$-embeddings of $\E$ into $\Mat_2(\F)$ are conjugated by some element in $\GL_2(\F)$, proving the existence of $a_r$. Since every $\F$-embedding obviously belongs to $\Op(L_r, \vO)$ for some $r \in \Z_{\geq 0}$, the group $\GL_2(\F)$ acts transitively from right on the disjoint union
	$$ \sideset{}{_{r \in \Z_{\geq 0}}} \bigsqcup \Op(L_r, \vO) = \iota^{\GL_2(\F)}. $$
	The stated double coset decomposition then follows readily from Theorem \ref{thm: OpEmbNr}, since the stabilizer subgroup of $\iota$ in $\GL_2(\F)$ is precisely $\iota(\E^{\times})$. For the ``moreover'' part note that $L_r^{\times} = \vo^{\times} U_{\E}^{er}$, implying
	$$ \frac{L_k^{\times}}{L_{k+r}^{\times}} = \frac{\vo^{\times} U_{\E}^{ek}}{\vo^{\times} U_{\E}^{e(k+r)}} \simeq \frac{U_{\E}^{ek}}{U_{\vo}^k U_{\E}^{e(k+r)}} \simeq \frac{U_{\E}^{ek}/U_{\E}^{e(k+r)}}{U_{\vo}^k U_{\E}^{e(k+r)} / U_{\E}^{e(k+r)}} \simeq \frac{U_{\E}^{ek}/U_{\E}^{e(k+r)}}{U_{\vo}^k / U_{\vo}^{k+r}}, \quad \forall \ k \geq 0, r \geq 1. $$
	Hence we deduce the equality
\begin{equation} \label{eq: QOrdUnitsQ}
	\extnorm{L_k^{\times}/L_{k+r}^{\times}} = \frac{\extnorm{U_{\E}^{ek}/U_{\E}^{e(k+r)}}}{\extnorm{U_{\vo}^k / U_{\vo}^{k+r}}} = 
	\begin{cases}
		q^r & \text{if } k > 0 \text{ or } e = 2 \\
		q^r(1+q^{-1}) & \text{if } k = 0 \text{ and } e = 1
	\end{cases}. 
\end{equation}
	The volume computation for $\oM$ follows easily from the case $k=0$ of \eqref{eq: QOrdUnitsQ}, and $\gp{K}_{\oM} = \varpi^{\Z} \oM^{\times}$, implying
	$$ \iota(\E^{\times}) \backslash \iota(\E^{\times}) a_r \gp{K}_{\oM} = \iota(\E^{\times}) \backslash \iota(\E^{\times}) a_r \oM^{\times} \simeq \iota^{a_r}(L_r^{\times}) \backslash \oM^{\times}. $$
	For $\oJ$ we read from Lemma \ref{lem: OpEmbFil} (2) that
	$$ \left( \oJ - \VP_{\oJ} \right) \cap \iota^{a_r}(\E) =
	\begin{cases}
		\iota^{a_r}(L_r - \varpi L_{r-1}) & \text{if } r \geq 1 \\
		\iota^{a_0}(\vO_{\E}^{\times}) & \text{if } r = 0
	\end{cases}, $$
	$$ \left( \VP_{\oJ} - \VP_{\oJ}^2 \right) \cap \iota^{a_r}(\E) =
	\begin{cases}
		\iota^{a_r}(\varpi L_{r-1} - \varpi L_r) & \text{if } r \geq 1 \\
		\iota^{a_0}(\varpi_{\E}\vO_{\E}^{\times}) & \text{if } r = 0
	\end{cases}. $$
	Now that $\oJ^{\times} \Pi^n = \left\{ x \in \VP_{\oJ}^n - \VP_{\oJ}^{n+1} \ \middle| \ \det(x) \in \varpi^n \vo^{\times} \right\}$, we deduce
	$$ \oJ^{\times} \cap \iota^{a_r}(\E) = 
	\begin{cases}
		\iota^{a_r}(L_r^{\times}) & \text{if } r \geq 1 \\
		\iota^{a_0}(\vO_{\E}^{\times}) & \text{if } r = 0
	\end{cases}, \quad \oJ^{\times} \Pi \cap \iota^{a_r}(\E) = 
	\begin{cases}
		\emptyset & \text{if } r \geq 1 \\
		\iota^{a_0}(\varpi_{\E}\vO_{\E}^{\times}) & \text{if } r = 0
	\end{cases} $$
	$$ \Rightarrow \quad \gp{K}_{\oJ} \cap \iota^{a_r}(\E) = 
	\begin{cases}
		\varpi^{\Z} \iota^{a_r}(L_r^{\times}) & \text{if } r \geq 1 \\
		\iota^{a_0}(\E^{\times}) & \text{if } r = 0
	\end{cases}. $$
	The volume computation for $\oJ$ follows again from the case $k=0$ of \eqref{eq: QOrdUnitsQ} and
	$$ \iota(\E^{\times}) \backslash \iota(\E^{\times}) a_r \gp{K}_{\oJ} = \left( \iota^{a_r}(\E) \cap \gp{K}_{\oJ} \right) \backslash \gp{K}_{\oJ} \simeq 
	\begin{cases}
		\left( \iota^{a_r}(L_r^{\times}) \backslash \oJ^{\times} \right) \times \left( \varpi^{\Z} \backslash \Pi^{\Z} \right) & \text{if } r \geq 1 \\
		\iota^{a_0}(\vO_{\E}^{\times}) \backslash \oJ^{\times} & \text{if } r = 0
	\end{cases}. $$
\end{proof}

\begin{remark}
	One may find an elementary proof of the decompositions in Lemma \ref{lem: IwaDVar} in \cite[Proposition 2.18]{Wu1}.
\end{remark}

\begin{lemma} \label{lem: OrbIntNSp}
	Consider quadratic field $\E/\F$ with ramification index $e=e(\E/\F)$ and regular elements $x \in \E^{\times}-\F^{\times}$. Normalize the measure with $\Vol(\vO_{\E}^{\times})=1$. We have the following formulae for $n \in \Z_{\geq 1}$
\begin{multline*}
	\vO(f_n; \E, x) =\left\{ \id_{U_{\E}^{en}}(x) + \left( \sum_{r=1}^{\infty} q^r \cdot \id_{\left( 1 + \varpi^n L_r \right) \cap L_r^{\times}}(x) \right) \cdot \frac{1-q^{-2}}{1-q^{-e}} \right\} \cdot \\
	\begin{cases}
		1 & \text{if } n = 0 \\
		q^{4n}(1-q^{-1})(1-q^{-2}) & \text{if } n \geq 1
	\end{cases};
\end{multline*}
\begin{multline*}
	\vO(g_{n}; \E, x) = \left\{ \id_{e=2} \id_{U_{\E}^{n}}(x) + \sum_{r=1}^{\infty} 2q^r \frac{1-q^{-2}}{1-q^{-e}} \cdot \id_{\left( 1 + \varpi^{\lceil \frac{n}{2} \rceil} L_{r-\id_{2 \nmid n}} \right) \cap L_{r-\id_{2 \nmid n}}^{\times}}(x) \right\} \cdot \\
	\begin{cases}
		1 & \text{if } n = 0 \\
		q^{2n}(1-q^{-1})^2 & \text{if } n \geq 1
	\end{cases}.
\end{multline*}
\end{lemma}
\begin{proof}
	 Note that $U_{\oM}^n \cap \iota^{a_r}(\E^{\times}) = \iota^{a_r}((1+\varpi^n L_r) \cap L_r^{\times})$ by Lemma \ref{lem: OpEmbFil} (2). Taking into account \eqref{eq: MUnitsInd} we deduce the first formula readily from Lemma \ref{lem: IwaDVar} via
	 $$ \vO(f_n; \E, x) = \sum_{r=0}^{\infty} \frac{\Vol\left( \iota(\E^{\times}) \backslash \iota(\E^{\times}) a_r \gp{K}_{\oM} \right)}{\Vol(\oM^{\times})} \cdot \id_{\left( 1 + \varpi^n L_r \right) \cap L_r^{\times}}(x) \cdot \begin{cases}
	 	1 & \text{if } n=0 \\
		q^{4n}(1-q^{-1})(1-q^{-2}) & \text{if } n \geq 1
	 \end{cases}. $$
	 Similarly by Lemma \ref{lem: OpEmbFil} (2) we have
	  $$ U_{\oJ}^n \cap \iota^{a_r}(\E^{\times}) = 
	  \begin{cases}
		\iota^{a_r}\left( (1+\varpi^k L_r) \cap L_r^{\times} \right) & \text{if } n=2k \\
		\iota^{a_r}\left( (1+\varpi^k L_{r-1}) \cap L_{r-1}^{\times} \right) & \text{if } n=2k-1, r \geq 1 \\
		\iota^{a_r}(U_{\E}^n) & \text{if } n=2k-1, r = 0
	\end{cases}. $$
	Taking into account \eqref{eq: JUnitsInd} we deduce the first formula readily from Lemma \ref{lem: IwaDVar} via
\begin{multline*} 
	\vO(g_{n}; \E, x) = \left\{ \id_{e=2} \cdot \frac{\Vol\left( \iota(\E^{\times}) \backslash \iota(\E^{\times}) a_0 \gp{K}_{\oJ} \right)}{\Vol(\oJ^{\times})} \cdot \id_{U_{\E}^{n}}(x) + \right. \\
	\left. \sum_{r=1}^{\infty} \frac{\Vol\left( \iota(\E^{\times}) \backslash \iota(\E^{\times}) a_r \gp{K}_{\oJ} \right)}{\Vol(\oJ^{\times})} \cdot \id_{\left( 1 + \varpi^{\lceil \frac{n}{2} \rceil} L_{r-\id_{2 \nmid n}} \right) \cap L_{r-\id_{2 \nmid n}}^{\times}}(x) \right\} \cdot 
	\begin{cases}
	 	1 & \text{if } n=0 \\
		q^{2n}(1-q^{-1})^2 & \text{if } n \geq 1
	 \end{cases}.
\end{multline*}
\end{proof}

\begin{proof}[Proof of Proposition \ref{prop: ArithMatch}]
	We check the conditions in Definition \ref{def: OrbIntM} (2). We readily verify the first condition (2.\rmnum{1}) by Lemma \ref{lem: OrbIntSp}. Applying Lemma \ref{lem: OrbIntNSp} and taking into account Lemma \ref{lem: OrbIntNSD} we get for non-split $\E$ and $n \geq 1$
\begin{align*}
	\vO(\widetilde{f}_0; \E, x) &= \tfrac{2}{e} \cdot \id_{\vO_{\E}^{\times}}(x) = \vO(\varphi_0; \E, x), \\
	\vO(\widetilde{f}_{2n}; \E, x) &= \tfrac{2}{e} \cdot \id_{U_{\E}^{en}}(x) \cdot q^{4n}(1-q^{-2}) = \vO(\varphi_{2n}; \E, x), \\
	\vO(\widetilde{f}_{2n-1}; \E, x) &= \tfrac{2}{e} \cdot \id_{U_{\E}^{\lceil \frac{e(2n-1)}{2} \rceil}}(x) \cdot q^{4n-2}(1-q^{-2}) = \vO(\varphi_{2n-1}; \E, x).
\end{align*}
	We have verified the second condition (2.\rmnum{2}) and conclude the stated matching.
\end{proof}

\section{Stable Orbital Integrals}

	Let $\B$ be a quaternion algebra over $\Q$ with reduced trace $\Tr = \Tr_{\B}$ and norm $\nu = \nu_{\B}$. Suppose $\B$ is unramified at $\infty$. Fix a maximal $\Z$-order $\vO_{\B}$ of $\B$. Suppose we have two (congruence) subgroups $\Gamma \triangleleft \widetilde{\Gamma} < \vO_{\B}^{\times}$, together with subgroups in finite adeles $\gp{K}_{\Gamma} \triangleleft \gp{K}_{\widetilde{\Gamma}} < \widehat{\vO}_{\B}^{\times} \subset \B^{\times}(\A_{\fin})$, where $\A_{\fin}$ is the ring of finite adeles of $\Q$, so that we have
	$$ \Gamma = \B^{\times}(\Q) \cap \gp{K}_{\Gamma}, \quad \widetilde{\Gamma} = \B^{\times}(\Q) \cap \gp{K}_{\widetilde{\Gamma}}, \quad \gp{K}_{\widetilde{\Gamma}} \cap \A_{\fin}^{\times} = \widehat{\Z}^{\times} $$
	and the following \emph{crucial assumption} as well as a consequence of it via \cite[Theorem 28.5.3]{Vo21} (the \emph{strong approximation theorem})
\begin{equation} \label{eq: LocGlobLatGps}
	 \nu(\gp{K}_{\widetilde{\Gamma}}) = \widehat{\Z}^{\times} \quad \Rightarrow \quad \B^{\times}(\A_{\fin}) = \B^{\times}(\Q) \cdot \gp{K}_{\widetilde{\Gamma}}.
\end{equation}
\begin{remark} \label{rmk: EgNiceLat}
	For example for $\B = \Mat_2$, the following groups satisfy the above assumptions ($N \in \Z_{\geq 1}$):
\begin{itemize}
	\item[(1)] $\Gamma = \left\{ \gamma \in \GL_2(\Z) \ \middle| \ \gamma - \id \in N \cdot \Mat_2(\Z) \right\}$, $\widetilde{\Gamma} = \GL_2(\Z)$, with $\gp{K}_{\Gamma} = \left\{ \kappa \in \GL_2(\widehat{\Z}) \ \middle| \ \kappa - \id \in N \cdot \Mat_2(\widehat{\Z}) \right\}$, $\gp{K}_{\widetilde{\Gamma}} = \GL_2(\widehat{\Z})$;
	\item[(2)] $\Gamma = \widetilde{\Gamma} = \begin{pmatrix} \Z & \Z \\ N \Z & \Z \end{pmatrix} \cap \GL_2(\Z)$, $\gp{K}_{\Gamma} = \gp{K}_{\widetilde{\Gamma}} = \begin{pmatrix} \widehat{\Z} & \widehat{\Z} \\ N \widehat{\Z} & \widehat{\Z} \end{pmatrix} \cap \GL_2(\widehat{\Z})$;
	\item[(3)] We can mix the construction of (1) and (2) locally.
\end{itemize}
\end{remark}
	
	Let $\B^1 < \B^{\times}$ be the sub-$\Q$-group of elements with reduced norm $1$. Let $\Gamma^1 := \Gamma \cap \B^1(\Q)$. Then we have $\Gamma^1 \triangleleft \widetilde{\Gamma}$. For any $\gamma \in \Gamma^1$ we introduce its conjugacy and \emph{stable} conjugacy class as
\begin{equation*}
	[\gamma] := \left\{ \sigma^{-1} \gamma \sigma \ \middle| \ \sigma \in \Gamma^1 \right\}, \quad [\gamma]_{\mathrm{st}} := \left\{ \sigma^{-1} \gamma \sigma \ \middle| \ \sigma \in \B^{\times}(\Q) \right\} \cap \Gamma^1.
\end{equation*}
	Clearly they satisfy the relation (by Skolem--Noether's theorem)
	$$ [\gamma]_{\mathrm{st}} = \sideset{}{_{\Tr(\gamma')=\Tr(\gamma)}} \bigsqcup [\gamma']. $$
	Since $\B_{\infty} \simeq \Mat_2(\R)$, the discrete group $\Gamma^1$, resp. $\pm \Gamma^1 / \{ \pm \id \}$ can be regarded as a lattice in $\SL_2(\R)$, resp. $\PSL_2(\R) = \SL_2(\R) / \{ \pm \id \}$. 
	
	Now consider $\widetilde{f}_{\infty} \in \Cont_c^{\infty}(\PGL_2(\R)//\gO_2(\R))$ ($=\Cont_c^{\infty}(\SL_2(\R)//\SO_2(\R))$) to be a bi-$\gO_2(\R)$-invariant smooth function on $\PGL_2(\R)$. Define an even function $\widetilde{g}_{\infty}$ on $\R$ by the formula
	$$ \int_{\R} \widetilde{f}_{\infty} \left( \begin{pmatrix} 1 & -x \\ & 1 \end{pmatrix} \begin{pmatrix} y & \\ & 1 \end{pmatrix} \begin{pmatrix} 1 & x \\ & 1 \end{pmatrix} \right) \ud x = \frac{\widetilde{g}_{\infty}(\log y)}{\sqrt{y}-\sqrt{y}^{-1}}, \quad \forall \ y > 1. $$
	Then we have for any hyperbolic $\gamma \in \Gamma^1$ with $N(\gamma)=x$ the following formula for orbital integral
\begin{equation} \label{eq: HypOrbInt}
	\int_{\Gamma^1 \R^{\times} \backslash \GL_2(\R)} \left( \sum_{\sigma \in [\gamma]} \widetilde{f}_{\infty}(g^{-1} \sigma g) \right) \ud g = \frac{\log x_0}{\sqrt{x} - \sqrt{x}^{-1}} \cdot \widetilde{g}_{\infty}(\log x), 
\end{equation}
	where $x_0 = x_0([\gamma]) = N(\gamma_0)$ for a \emph{primitive conjugacy class} associated with $\gamma$. Note that the formula \eqref{eq: HypOrbInt} is a group-theoretic reformulation of \cite[Proposition \Rmnum{1}.6.3]{He76}. In particular, $\widetilde{g}_{\infty}$ is the Harish-Chandra transform of $\widetilde{f}_{\infty}$ and can be any even function in $\Cont_c^{\infty}(\R)$ by \cite[Proposition 3.4]{KL13}. For any $t \in \R$ we define
\begin{equation} \label{eq: DiffPGTF}
	\psi_{\Gamma^1}(t) := \frac{1}{1 - x^{-1}} \sum_{[\gamma]: \Tr(\gamma)=t} \log x_0([\gamma]).
\end{equation}
	The sum is also over conjugacy classes contained in a (the only) stable conjugacy class $[\gamma]_{\mathrm{st}}$ with $\Tr(\gamma)=t$.  Since $N(\gamma) = \left( \norm[t] + \sqrt{t^2-4} \right)/2 =: x(t)$ for hyperbolic $\gamma$, we get
\begin{equation} \label{eq: PGTFRel}
	\Psi_{\Gamma^1}(x) = c_{\Gamma^1} \cdot \sum_{2< \norm[t] \leq x^{1/2} + x^{-1/2}} \psi_{\Gamma^1}(t), \quad c_{\Gamma^1} := \frac{[\pm \Gamma^1 : \Gamma^1]}{2} = 
	\begin{cases}
		1 & \text{if } -\id_{\B} \notin \Gamma^1 \\
		1/2 & \text{if } -\id_{\B} \in \Gamma^1
	\end{cases}.
\end{equation}
	The extra factor $c_{\Gamma^1}$ takes into account that we should consider elements in $\PSL_2(\R)$ instead of $\SL_2(\R)$. The sum, being essentially over the so-called \emph{length spectra} of $\Gamma^1 \backslash \ag{H}$, is discrete. 
	
	We propose to study the following \emph{stable} orbital integral (for any hyperbolic $\gamma \in \Gamma^1$ with $\Tr(\gamma)=t$)
\begin{equation} \label{eq: StOrbInt}
	I(t; \widetilde{f}_{\infty}) := \int_{\Gamma^1 \R^{\times} \backslash \GL_2(\R)} \left( \sum_{\sigma \in [\gamma]_{\mathrm{st}}} \widetilde{f}_{\infty}(g^{-1} \sigma g) \right) \ud g = \frac{\widetilde{g}_{\infty}(\log x)}{\sqrt{x}} \cdot \psi_{\Gamma^1} \left( t \right), \quad x = \frac{\norm[t]+\sqrt{t^2-4}}{2}.
\end{equation}
\begin{proposition} \label{prop: StAdeleOrbInt}
	(0) The crucial assumption \eqref{eq: LocGlobLatGps} implies $\B^{\times}(\A_{\fin}) = \B^{\times}(\Q) \cdot \gp{K}_{\widetilde{\Gamma}}$.
	
\noindent (1) Let $\gamma \in \Gamma^1$ be hyperbolic with $\Tr(\gamma) = t$ (hence $\norm[t] > 2$). Such $\gamma$ is $\Q$-elliptic, i.e., the $\Q$-algebra $\Q[\gamma]$ is a (real) quadratic field.
	
\noindent (2) We have the equation
	$$ I(t; \widetilde{f}_{\infty}) = \frac{1}{c_{\Gamma^1}} \frac{\Vol(\widehat{\Z}^{\times})}{[\widehat{\Z}^{\times}:\nu(\gp{K}_{\Gamma})]} \frac{\Vol \left( \gp{Z}_{\gamma}(\Q) \A^{\times} \backslash \gp{Z}_{\gamma}(\A) \right)}{\Vol(\gp{K}_{\Gamma})} \int_{\gp{Z}_{\gamma}(\A) \backslash \B^{\times}(\A)} \left( \widetilde{f}_{\infty} \otimes \id_{\gp{K}_{\Gamma}} \right)(g^{-1} \gamma g) \ud g, $$
	where $\R^{\times} S$ is the image of a subset $S \subset \GL_2(\R)$ under the projection $\GL_2(\R) \to \PGL_2(\R)$, and $\gp{Z}_{\gamma}$ is the centralizer of $\gamma$ in $\B^{\times}(\Q)$ regarded as an algebraic group defined over $\Q$.
\end{proposition}
\begin{proof}
	(0) This follows directly from the \emph{strong approximation theorem} \cite[Theorem 28.5.3]{Vo21}.

\noindent (1) This is a generalization of \cite[Lemma 2.8]{CWZ21}. The proof given there remains valid.

\noindent (2) Note that the integrand is a function in $g$ left invariant by $\widetilde{\Gamma}$, since $[\gamma]_{\mathrm{st}}$ is stable by conjugation by $\widetilde{\Gamma}$. Hence we get
	$$ I(t; \widetilde{f}_{\infty}) = [\R^{\times} \widetilde{\Gamma} : \R^{\times} \Gamma^1] \cdot \int_{\widetilde{\Gamma} \R^{\times} \backslash \GL_2(\R)} \left( \sum_{\sigma \in [\gamma]_{\mathrm{st}}} \widetilde{f}_{\infty}(g^{-1} \sigma g) \right) \ud g. $$
	For any $\sigma \in \B^{\times}(\Q)$ we have $\nu(\sigma^{-1} \gamma \sigma) = \nu(\gamma) = 1$. We deduce
	$$ \sigma^{-1} \gamma \sigma \in [\gamma]_{\mathrm{st}} \ \Leftrightarrow \ \sigma^{-1} \gamma \sigma \in \Gamma \ \Leftrightarrow \ \sigma^{-1} \gamma \sigma \in \gp{K}_{\Gamma}. $$
	Therefore we can rewrite (in the second line $\widetilde{\Gamma}$ can be viewed as subgroup of either $\GL_2(\R)$ or $\B^{\times}(\A)$)
\begin{multline*}
	I(t; \widetilde{f}_{\infty}) = [\R^{\times} \widetilde{\Gamma} : \R^{\times} \Gamma^1] \cdot \int_{\widetilde{\Gamma} \R^{\times} \backslash \GL_2(\R)} \left( \sum_{\sigma \in \gp{Z}_{\gamma}(\Q) \backslash \B^{\times}(\Q)} \widetilde{f}_{\infty}(g^{-1} \sigma^{-1} \gamma \sigma g) \cdot \id_{\gp{K}_{\Gamma}}(\sigma^{-1} \gamma \sigma) \right) \ud g \\
	= \frac{[\R^{\times} \widetilde{\Gamma} : \R^{\times} \Gamma^1]}{\Vol(\A_{\fin}^{\times} \backslash \A_{\fin}^{\times}\gp{K_{\widetilde{\Gamma}}})} \cdot \int_{\widetilde{\Gamma} \A^{\times} \backslash \left( \GL_2(\R) \times \A_{\fin}^{\times}\gp{K}_{\widetilde{\Gamma}} \right)} \left( \sum_{\sigma \in \gp{Z}_{\gamma}(\Q) \backslash \B^{\times}(\Q)} \left( \widetilde{f}_{\infty} \otimes \id_{\gp{K}_{\Gamma}} \right)(g^{-1} \sigma^{-1} \gamma \sigma g) \right) \ud g.
\end{multline*}
	Recall $\B^{\times}(\A_{\fin}) = \B^{\times}(\Q) \cdot \gp{K}_{\widetilde{\Gamma}}$, $\widehat{\Z}^{\times} < \gp{K}_{\widetilde{\Gamma}}$ by assumption and $\A_{\fin}^{\times} = \Q^{\times} \widehat{\Z}^{\times}$. We deduce $\B^{\times}(\Q) \cap \A_{\fin}^{\times}\gp{K}_{\widetilde{\Gamma}} = \Q^{\times} \left( \B^{\times}(\Q) \cap \gp{K}_{\widetilde{\Gamma}} \right) = \Q^{\times} \widetilde{\Gamma}$ or $\B^{\times}(\Q) \A_{\fin}^{\times} \cap \A_{\fin}^{\times}\gp{K}_{\widetilde{\Gamma}} = \widetilde{\Gamma} \A_{\fin}^{\times}$ and continue to rewrite $I(t; \widetilde{f}_{\infty})$ as
\begin{multline*}
	 \frac{[\R^{\times} \widetilde{\Gamma} : \R^{\times} \Gamma^1]}{\Vol(\A_{\fin}^{\times} \backslash \A_{\fin}^{\times}\gp{K_{\widetilde{\Gamma}}})} \cdot \int_{\B^{\times}(\Q) \A^{\times} \backslash \left( \GL_2(\R) \times \B^{\times}(\A_{\fin}) \right)} \left( \sum_{\sigma \in \gp{Z}_{\gamma}(\Q) \backslash \B^{\times}(\Q)} \left( \widetilde{f}_{\infty} \otimes \id_{\gp{K}_{\Gamma}} \right)(g^{-1} \sigma^{-1} \gamma \sigma g) \right) \ud g \\
	= \frac{[\R^{\times} \widetilde{\Gamma} : \R^{\times} \Gamma^1]}{\Vol(\A_{\fin}^{\times} \backslash \A_{\fin}^{\times}\gp{K_{\widetilde{\Gamma}}})} \cdot \int_{\gp{Z}_{\gamma}(\Q) \A^{\times} \backslash \B^{\times}(\A)} \left( \widetilde{f}_{\infty} \otimes \id_{\gp{K}_{\Gamma}} \right)(g^{-1} \gamma g) \ud g \\
	= [\R^{\times} \widetilde{\Gamma} : \R^{\times} \Gamma^1] \frac{\Vol(\gp{K_{\Gamma}})}{\Vol(\A_{\fin}^{\times} \backslash \A_{\fin}^{\times}\gp{K_{\widetilde{\Gamma}}})} \cdot \frac{\Vol \left( \gp{Z}_{\gamma}(\Q) \A^{\times} \backslash \gp{Z}_{\gamma}(\A) \right)}{\Vol(\gp{K}_{\Gamma})} \int_{\gp{Z}_{\gamma}(\A) \backslash \B^{\times}(\A)} \left( \widetilde{f}_{\infty} \otimes \id_{\gp{K}_{\Gamma}} \right)(g^{-1} \gamma g) \ud g.
\end{multline*}
	If we denote by $\gp{K}_{\widetilde{\Gamma}}^1 := \gp{K}_{\widetilde{\Gamma}} \cap \B^1(\A_{\fin})$ and $\gp{K}_{\Gamma}^1 := \gp{K}_{\Gamma} \cap \B^1(\A_{\fin})$, then we have a short exact sequence of (finite) groups and an equation
	$$ 1 \to \gp{K}_{\widetilde{\Gamma}}^1/\gp{K}_{\Gamma}^1 \to \gp{K}_{\widetilde{\Gamma}}/\gp{K}_{\Gamma} \xrightarrow{\nu} \widehat{\Z}^{\times}/\nu(\gp{K}_{\Gamma}) \to 1 \quad \Rightarrow \quad \frac{\Vol(\gp{K_{\Gamma}})}{\Vol(\gp{K_{\widetilde{\Gamma}}})} = \frac{\Vol(\gp{K_{\Gamma}}^1)}{\Vol(\gp{K_{\widetilde{\Gamma}}}^1)} \cdot \frac{1}{[\widehat{\Z}^{\times} : \nu(\gp{K}_{\Gamma})]}. $$
	The strong approximation in $\B^1$ \cite[Theorem 28.5.3]{Vo21}, implying $\B^1(\A_{\fin}) = \B^1(\Q) \gp{K}_{\widetilde{\Gamma}}^1 = \B^1(\Q) \gp{K}_{\Gamma}^1$, yields
	$$ \gp{K}_{\widetilde{\Gamma}}^1/ \gp{K}_{\Gamma}^1 \simeq \left( \gp{K}_{\widetilde{\Gamma}}^1 \cap \B^1(\Q) \right) / \left( \gp{K}_{\Gamma}^1 \cap \B^1(\Q) \right) = \widetilde{\Gamma}^1 / \Gamma^1 \quad \Rightarrow \quad \frac{\Vol(\gp{K_{\Gamma}}^1)}{\Vol(\gp{K_{\widetilde{\Gamma}}}^1)} = \frac{1}{[\widetilde{\Gamma}^1 : \Gamma^1]}. $$
	We then conclude the stated equation by
	$$ [\R^{\times} \widetilde{\Gamma} : \R^{\times} \Gamma^1] \frac{\Vol(\gp{K_{\Gamma}})}{\Vol(\A_{\fin}^{\times} \backslash \A_{\fin}^{\times}\gp{K_{\widetilde{\Gamma}}})} = \frac{[\widetilde{\Gamma} : \pm \Gamma^1]}{[\widetilde{\Gamma}^1 : \Gamma^1]} \frac{\Vol(\widehat{\Z}^{\times})}{[\widehat{\Z}^{\times}:\nu(\gp{K}_{\Gamma})]} = \frac{1}{c_{\Gamma^1}} \frac{\Vol(\widehat{\Z}^{\times})}{[\widehat{\Z}^{\times}:\nu(\gp{K}_{\Gamma})]}, $$
	where we have used $[\widetilde{\Gamma} : \widetilde{\Gamma}^1] = \norm[\nu(\widetilde{\Gamma})] = \norm[\Z^{\times}] = 2$.
\end{proof}	

\begin{remark} \label{rmk: B1NotB}
	Note that $[\widehat{\Z}^{\times}:\nu(\gp{K}_{\Gamma})] \cdot \Vol(\gp{K}_{\Gamma}) = \Vol \left( \gp{K}_{\Gamma}^1 \right)$, and the formula in Proposition \ref{prop: StAdeleOrbInt} (2) only uses the values of $\widetilde{f}_{\infty} \otimes \id_{\gp{K}_{\Gamma}}$ on $\B^1(\A)$. Hence this formula is essentially for $\B^1$, although we have written it for a test function on $\B^{\times}(\A)$.
\end{remark}

\section{Proof of Main Theorem}

	Let $\Ram(\D)$ be the set of ramified places of $\D$. Given $\idlN$, there is a sequence of $n_p \in \Z_{\geq 0}$, with $n_p=0$ for all but finitely many $p$, so that the completion $\idlN_p$ of $\idlN$ at $p$ is the $n_p$-th power of the Jacobson radical of the completion $\vO_{\D,p}$ of $\vO_{\D}$ at primes $p$. Write the functions $f_n$, resp. $g_n$, resp. $\varphi_n$ in Definition \ref{def: PrinpChF} (3) for $\F = \Q_p$ as $f_{n,p}$, resp. $g_{n,p}$, resp. $\varphi_{n,p}$. Define
	$$ \varphi_{\D} := \widetilde{f}_{\infty} \bigotimes_{p \in \Ram(\D)} \varphi_{n_p,p} \bigotimes_{p \notin \Ram(\D)} f_{n_p,p}. $$
	It follows from Proposition \ref{prop: StAdeleOrbInt} that
\begin{equation} \label{eq: StQuat}
	c_{\Gamma_{\D}^1(\idlN)} \cdot \frac{\widetilde{g}_{\infty}(\log x)}{\sqrt{x}} \cdot \psi_{\Gamma_{\D}^1(\idlN)} \left( t \right) = \\
	\Vol(\widehat{\Z}^{\times}) \Vol \left( \gp{Z}_{\gamma}(\Q) \A^{\times} \backslash \gp{Z}_{\gamma}(\A) \right) \int_{\gp{Z}_{\gamma}(\A) \backslash \D^{\times}(\A)} \varphi_{\D}(g^{-1} \gamma g) \ud g.
\end{equation}
	Write the local matching in Proposition \ref{prop: ArithMatch} as
	$$ \varphi_{n_p,p} \leftrightarrow a_p \cdot f_{n_p^1,p} + b_p \cdot g_{n_p^2,p}. $$
	We get the global matching of test functions
	$$ \varphi_{\D} \leftrightarrow \sum_{I \subset \Ram(\D)} \left( \prod_{p \in I} a_p \prod_{p \in \Ram(\D)-I} b_p \right) \cdot f_I, \quad f_I := \widetilde{f}_{\infty} \bigotimes_{p \in I} f_{n_p^1,p} \bigotimes_{p \in \Ram(\D)-I} g_{n_p^2,p} \bigotimes_{p \notin \Ram(\D)} f_{n_p,p}. $$
	For each $I \subset \Ram(\D)$, let $\vO_I$ be the Eichler order, whose completion $\vO_{I,p}$ at $p \in I$ is $\oM_p := \Mat_2(\Z_p)$ and at $p \in \Ram(\D)-I$ is $\oJ_p := \begin{pmatrix} \Z_p & \Z_p \\ p \Z_p & \Z_p \end{pmatrix}$. Let $\idlN_I$ be the two sided ideal of $\vO_I$, whose completion at $p \in I$, resp. $p \notin \Ram(\D)$, resp. $p \in \Ram(\D)-I$ is the $n_p^1$-th, resp. $n_p$-th, resp. $n_p^2$-th power of the Jacobson radical of $\vO_{I,p}$. Define
	$$ \Gamma^1(\idlN_I) := \left\{ x \ \middle| \ x \in \vO_I^{\times}, \det(x)=1, x - \id \in \idlN_I \right\}, \quad \Gamma(\idlN_I) := \{ \pm \id \} \cdot \Gamma^1(\idlN_I). $$
	It follows from Proposition \ref{prop: StAdeleOrbInt} (taking into account Remark \ref{rmk: EgNiceLat} (3)) again that
\begin{equation} \label{eq: StMat}
	c_{\Gamma^1(\idlN_I)} \cdot \frac{\widetilde{g}_{\infty}(\log x)}{\sqrt{x}} \cdot \psi_{\Gamma^1(\idlN_I)} \left( t \right) = \Vol(\widehat{\Z}^{\times}) \Vol \left( \gp{Z}_{\gamma}(\Q) \A^{\times} \backslash \gp{Z}_{\gamma}(\A) \right) \int_{\gp{Z}_{\gamma}(\A) \backslash \GL_2(\A)} f_I(g^{-1} \gamma g) \ud g.
\end{equation}
	Inserting \eqref{eq: StQuat} and \eqref{eq: StMat} into the equation for the global matching integrals
\begin{equation} \label{eq: GlobMatch}
	\int_{\gp{Z}_{\gamma}(\A) \backslash \D^{\times}(\A)} \varphi_{\D}(g^{-1} \gamma g) \ud g = \sum_{I \subset \Ram(\D)} \left( \prod_{p \in I} a_p \prod_{p \in \Ram(\D)-I} b_p \right) \cdot \int_{\gp{Z}_{\gamma}(\A) \backslash \GL_2(\A)} f_I(g^{-1} \gamma g) \ud g,
\end{equation}
	we obtain
\begin{equation*}
	c_{\Gamma_{\D}^1(\idlN)} \cdot \frac{\widetilde{g}_{\infty}(\log x)}{\sqrt{x}} \cdot \psi_{\Gamma_{\D}^1(\idlN)} \left( t \right) = \sum_{I \subset \Ram(\D)} \left( \prod_{p \in I} a_p \prod_{p \in \Ram(\D)-I} b_p \right) c_{\Gamma^1(\idlN_I)} \cdot \frac{\widetilde{g}_{\infty}(\log x)}{\sqrt{x}} \cdot \psi_{\Gamma^1(\idlN_I)} \left( t \right),
\end{equation*}
	or equivalently by the arbitrariness of $\widetilde{g}_{\infty} \in \Cont_c^{\infty}(\R)$ (and even)
\begin{equation} \label{eq: dPGTEq}
	c_{\Gamma_{\D}^1(\idlN)} \cdot \psi_{\Gamma_{\D}^1(\idlN)} \left( t \right) = \sum_{I \subset \Ram(\D)} \left( \prod_{p \in I} a_p \prod_{p \in \Ram(\D)-I} b_p \right) c_{\Gamma^1(\idlN_I)} \cdot \psi_{\Gamma^1(\idlN_I)} \left( t \right).
\end{equation}
	Summing \eqref{eq: dPGTEq} over $t$ we get by \eqref{eq: PGTFRel}
\begin{equation} \label{eq: PsiRel}
	\Psi_{\Gamma_{\D}(\idlN)}(x) = \sum_{I \subset \Ram(\D)} \left( \prod_{p \in I} a_p \prod_{p \in \Ram(\D)-I} b_p \right) \Psi_{\Gamma(\idlN_I)}(x).
\end{equation}
	Now that the prime geodesic theorem holds for $\Gamma(\idlN_I)$ by \cite[Corollary 1.2]{LRS95} in the form
\begin{equation} \label{eq: PGTMat}
	\Psi_{\Gamma(\idlN_I)}(x) = x + O_{\epsilon}(x^{\frac{7}{10}+\epsilon}),
\end{equation}
	we conclude the proof of Theorem \ref{thm: Main} for $\Psi_{\Gamma_{\D}(\idlN)}$ by the crucial identity
	$$ \sum_{I \subset \Ram(\D)} \left( \prod_{p \in I} a_p \prod_{p \in \Ram(\D)-I} b_p \right) = \prod_{p \in \Ram(\D)} (a_p+b_p) = 1, $$
	since we have $a_p+b_p=1$ by observing Proposition \ref{prop: ArithMatch}. Finally to derive the statement for $\pi_{\Gamma_{\D}(\idlN)}$ we simply do an integration by parts (or abel summation).

%\bibliographystyle{acm}
%	
%\bibliography{mathbib}

\end{document}